\documentclass[12pt]{amsart}
\usepackage{amsfonts}
\usepackage{amsmath}
\usepackage{amssymb,latexsym}
\usepackage[mathcal]{eucal}
\usepackage{amscd}

\usepackage[pdftex,bookmarks,colorlinks,breaklinks]{hyperref}
\input xy
\xyoption{all}

\newcommand{\Fcal}{\mathcal{F}}
\newcommand{\Gcal}{\mathcal{G}}

\newcommand{\Tcal}{\mathcal{T}}

\newcommand{\Z}{\mathbb{Z}}

\newcommand{\N}{\mathbb{N}}

\newcommand{\al}{\alpha}

\newcommand{\del}{\delta}
\newcommand{\Del}{\Delta}
\newcommand{\ep}{\epsilon}

\newcommand{\ol}{\overline}

\newcommand{\br}{\vspace{3 mm}}
\newcommand{\imp}{\Rightarrow}

\newcommand{\cls}{{\rm{cls\,}}}

\newcommand{\id}{{\rm{id}}}

\newcommand{\Homeo}{\rm{Homeo\,}}

\swapnumbers
\theoremstyle{plain}
\newtheorem{thm}{Theorem}[section]
\newtheorem{cor}[thm]{Corollary}
\newtheorem{lem}[thm]{Lemma}
\newtheorem{prop}[thm]{Proposition}

\theoremstyle{definition}
\newtheorem{defn}[thm]{Definition}

\newtheorem{rmk}[thm]{Remark}

\numberwithin{thm}{section}

\begin{document}

\title[]
{
$RP^{[d]}$ is an equivalence relation\\ 
An enveloping semigroup proof}

\author{Eli Glasner}

\address{Department of Mathematics\\
     Tel Aviv University\\
         Tel Aviv\\
         Israel}
\email{glasner@math.tau.ac.il}

\thanks{I thank the Hausdorff Institute at Bonn for the opportunity to participate in the 
Program ``Universality and Homogeneity" where this work was written,
November 2013.}



\begin{date}
{December 9, 2013}
\end{date}

\begin{abstract}
We present a purely enveloping semigroup proof of a theorem of Shao and Ye which asserts that for an abelian group $T$, a minimal flow $(X,T)$
and any integer $d \ge 1$,
the regional proximal relation of order $d$ is an equivalence relation.
\end{abstract}

\maketitle

\section*{}
Let $T$ be a countable abelian group and let $(X,T)$ be a minimal
flow; i.e. $X$ is a compact Hausdorff space and $T$ acts
on it as a group of homeomorphisms in such a way that for each $x \in X$
its $T$-orbit, $Tx =\{tx : t \in T\}$, is dense in $X$. 
Following \cite{HKM} and \cite{SY} we introduce the following notations
(generalizing from the case $T = \Z$ to the case of a general $T$ action).
For an integer $d \ge 1$ let $X^{[d]} = X^{2^d}$. We index the coordinates of an element
$x \in X^{[d]}$ by subsets $\ep \subset \{1,\dots,d\}$. Thus
$x = (x_\ep : \ep \subset \{1,\dots,d\})$, where for each $\ep$, $x_\ep \in X_\ep = X$. 
E.g. for $d =2$ we have $x = (x_\emptyset, x_{\{1\}}, x_{\{2\}}, x_{\{1,2\}})$.

Let $\pi_* : X^{[d]} \to X^{2^d -1}$ denote the projection onto the
last $2^d-1$-coordinates; i.e., the map which forgets the $\emptyset$-coordinate.
Let $X_*^{[d]} = \pi_*(X^{[d]}) = X^{2^d -1} = \prod \{X_\ep : \ep \not = \emptyset\}$
and for $x \in X^{[d]}$ let $x_* = \pi_*(x) \in X_*^{[d]}$ denote its projection; i.e.
$x_*$ is obtained by omitting the $\emptyset$-coordinate of $x$.
For each $\ep \subset \{1,\dots,d\}$ we denote by $\pi_\ep$ the projection map 
from $X^{[d]}$ onto $X_\ep =X$.
For a point $x \in X$ we let $x^{[d]} \in X^{[d]}$ and $x_*^{[d]} \in X_*^{[d]}$
be the {\em diagonal points} all of whose coordinates are $x$.
$\Del^{[d]} =\{x^{[d]}: x \in X\}$ is the {\em diagonal} of $X^{[d]}$
and $\Del_*^{[d]} =\{x_*^{[d]}: x \in X\}$ the {\em diagonal} of $X_*^{[d]}$.
Another convenient representation of $X^{[d]}$ is as a product space
$X^{[d]} = X^{[d-1]} \times X^{[d-1]}$ (with $X^{[0]} = X$).
When using this decomposition we write $x =(x',x'')$. 
More explicitly, for
for $\ep \subset \{1, \dots, d -1\}$ let $\ep d = \ep \cup \{d\}$, and
define the identification
$X^{[d-1]}\times X^{[d-1]} \to X^{[d]}$ by $(x',x'') \mapsto x$ 
with $x_\ep = x'_\ep$ and $x_{\ep d} = x''_{\ep}$. 
We will refer to 
$x'$ and $x''$ as the first and second $2^{d-1}$ coordinates, respectively.

We next define two group actions on $X^{[d]}$,
the {\em face group action} $\Fcal_d$ and the {\em total group action} $\Gcal_d$.
These actions are representations of $T^d = T \times T \times \cdots \times T$ ($d$ times) and 
$T^{d+1}$, respectively,
as subgroups of ${\Homeo}(X^{[d]})$.
For the $\Fcal_d$ action, 
$\Fcal_d  \times X^{[d]} \to X^{[d]}$,
$$
((t_1, \dots ,t_d),\ (x_\ep : \ep \subset \{1,\dots ,d\})) 
\mapsto (t_\ep x_\ep : \ep \subset \{1,\dots ,d\}), 
$$
where 
$t_\ep x_\ep  = t_{n_k} \cdots t_{n_1} x_\ep,
\ {\text{if}}\ \ep=\{n_1,\dots ,n_k\}$ and $t_\emptyset x_\emptyset = x_\emptyset$.
We can then represent the homeomorphism $\tau \in \Fcal_d$ which corresponds to 
$(t_1, \dots ,t_d) \in T^d$
as 
$$
\tau= \tau^{[d]}_{(t_1, \dots ,t_d)} = (t_\ep : \ep \subset \{1,\dots,d\}).
$$
We will also consider the restriction of the $\Fcal_d$ action to $X_*^{[d]}$ which is defined by omitting the $\emptyset$ coordinate.
Note that under the action of $\Fcal_d$ on $X^{[d]}$  the 
$\emptyset$-coordinate is left fixed.

For example, if we consider a minimal cascade $(X,f)$, taking $T = \Z = \{f^n : n \in \Z\}, d =3$ 
and $\tau = \tau^{[3]}_{(2,5,11)} \in \Fcal_3 \cong \Z^3$, we have:
$$
\tau(x) =
(x_\emptyset, f^2x_{\{1\}}, f^5 x_{\{2\}}, f^{2 +5} x_{\{1,2\}}, f^{11} x_{\{3\}}, f^{2 +11} x_{\{1,3\}},
f^{2 + 5 + 11} x_{\{1,2,3\}}),
$$ 
and
$$   
\tau(x_*) =
(f^2x_{\{1\}}, f^5 x_{\{2\}}, f^{2 +5} x_{\{1,2\}}, f^{11} x_{\{3\}}, f^{2 +11} x_{\{1,3\}},
f^{2 + 5 + 11} x_{\{1,2,3\}}).
$$

Note that the fact that the $\Fcal_d$ action is well defined depends on the 
commutativity of the group $T$. 

The action of $T^{d+1}$ on $X^{[d]}$, denoted by $\Gcal_d$, is the action generated by
the face group action $\Fcal_d$ and the {\em diagonal $\theta$-action} of $T$,
$T \times X^{[d]} \to X^{[d]}$, defined by
$$
(t, x) \mapsto \theta^{[d]}_t x = (tx_\ep : \ep \subset \{1,\dots,d\}).
$$
Thus for the $\Gcal_d$ action 
$\Gcal_d  \times X^{[d]} \to X^{[d]}$,
$$
((t_1, \dots ,t_d, t_{d+1}),\ (x_\ep : \ep \subset \{1,\dots ,d\})) 
\mapsto (t_{d+1}t_\ep x_\ep : \ep \subset \{1,\dots ,d\}), 
$$
where 
$t_\ep x_\ep  = t_{n_k} \cdots t_{n_1} x_\ep,
\ {\text{if}}\ \ep=\{n_1,\dots ,n_k\}$ and $t_\emptyset x_\emptyset = x_\emptyset$.
In other words, the $\Gcal_d$ action on $X^{[d]}$ is given by the representation:
$$
T^{d+1} \to {\Homeo}(X^{[d]}), \qquad (t_1,\dots,t_d, t_{d+1}) \mapsto 
\theta^{[d]}_{t_{d+1}} \tau^{[d]}_{(t_1, \dots ,t_d)}.
$$

Notice that 
\begin{equation}\label{tautheta}
\tau^{[d]}_{(t_1, \dots ,t_d)} (x',x'') = (\tau^{[d-1]}_{(t_1, \dots ,t_{d-1})} x',
\theta^{[d-1]}_{t_d} \tau^{[d-1]}_{(t_1, \dots ,t_{d-1})} x'').
\end{equation}

\br

In their paper \cite{SY} Shao and Ye prove that $RP^{[d]}$, the generalized  
regionally proximal relation of order $d$, is always an equivalence relation for a minimal
cascade $(X,T)$.
Their proof is based on the detailed analysis of the $\Gcal_d$ action provided by Host Kra and Mass
in \cite{HKM},  where the authors treated the distal case. 
The main tool used by Shao and Ye is a theorem which asserts
that, for each $x \in X$, the action of the face group $\Fcal_d$ on
the orbit closure $\cls \Fcal_d x_*^{[d]}$ is minimal. 
Their proof of this latter theorem was based on the general structure theory of minimal flows due to
Ellis-Glasner-Shapiro \cite{EGS}, McMahon \cite{Mc} and Veech \cite{V}. 
Now, it turns out that there is a direct enveloping semigroup proof of this theorem 
which is very similar to the proof by Ellis and Glasner given in \cite[page 46]{G}. 
The possibility of applying the Ellis-Glasner proof as a shortcut to 
Shao and Ye's proof was also discovered by Ethan Akin.
In the next section I present this short proof, established for a general commutative group.
For the interested reader I will, in a subsequent section, reproduce the beautiful proof of Shao
and Ye of the fact that for each $d \ge 1$, $RP^{[d]}$ is an equivalence relation.

Let us note that all the results of this work extend to the case where
$T$ is  a commutative separable topological group which acts continuously and  minimally on a compact metric space $X$. In fact, both 
minimality and the regionally proximal relations are the same for the group and for a countable dense subgroup.
Moreover, most of the results (like Theorem \ref{min}, parts (1) - (3), as well
as Theorem \ref{RP}) hold for actions of $T$ on a general compact space
(not necessarily metrizable).
I wish to thank Ethan Akin whose suggestions led to improvements of a first draft of this work.

\br

\section{The minimality of the face action on $Q_{x_*}^{[d]}$}

Let $(X,T)$ be a minimal flow with $T$ abelian. 
%
Let 
\begin{gather*}
Q^{[d]}= \cls \{g x^{[d]} : x \in X, \ g \in \Gcal_d\} = 
\ol{\Gcal_d\Del^{[d]}}=
\ol{\Fcal_d\Del^{[d]}},\\
Q_x^{[d]} = Q^{[d]} \cap (\{x\} \times X^{2^d -1}), \ \ {\text{and}}\ \ 
Q_{x_*}^{[d]} = \pi_*(Q_x^{[d]}).
\end{gather*}
For each $x \in X$ let $Y_x^{[d]} = \ol{\Fcal_d(x^{[d]})} \subset Q_x^{[d]}$
be the orbit closure of $x^{[d]}$ under $\Fcal_d$.
Finally, 
let $Y_{x_*}^{[d]} = \pi_*(Y_x^{[d]})$.
 
\begin{thm}[Shao and Ye]\label{min}
\begin{enumerate}
\item
The flow $(Q^{[d]},\Gcal_d)$ is minimal.
\item
For each $x \in X$, the flows $(Y_x^{[d]},\Fcal_d)$, and hence also
$(Y_{x_*}^{[d]},\Fcal_d)$, are minimal.
\item
For each $x \in X$ the set
$Y_x^{[d]}$ is the unique minimal subflow of the $\Fcal_d$-flow 
$(Q_x^{[d]}, \Fcal_d)$. Hence also
$Y_{x_*}^{[d]}$ is the unique minimal subflow of the $\Fcal_d$-flow 
$(Q_{x_*}^{[d]}, \Fcal_d)$.
\item
\footnote{This seems to be a new observation.}
For a dense $G_\del$ subset $X_0 \subset X$ we have $Y_x^{[d]} = Q_x^{[d]}$.
\end{enumerate}
\end{thm}

\begin{proof}
1. \ 
Let us denote $N := Q^{[d]}$ and $\Tcal : = \Gcal_d$.
Let $E = E(N, \Tcal)$ be the enveloping semigroup of $(N, \Tcal)$. 
Let $\pi_\ep : N \to X_\ep = X$ be the projection of $N$ on the $\ep$ coordinate, where 
$\ep\subset \{1,...,d\}$. 
We consider the action of the group $\Tcal$ on the $\ep$ coordinate via the projection $\pi_\ep$, 
that is, for $\ep  \subset \{1,\dots,d\}, (t_1,\dots ,t_d, t_{d+1}) \in T^{d+1}$ and $x \in X_\ep = X$,
$$
\Tcal \times X_\ep \to X_\ep, \quad  (\theta^{[d]}_{t_{d+1}}\tau^{[d]}_{(t_1,\dots ,t_d)},x)
\mapsto t_{d+1}t_\ep x.
$$
With respect to this action of $\Tcal$ on $X_\ep = X$ the map 
$\pi_\ep : (N,\Tcal) \to (X_\ep,\Tcal)$ is a flow homomorphism. 
Let $\pi_\ep^\bullet : E(N,\Tcal) \to E(X_\ep, \Tcal)$ be the corresponding homomorphism of enveloping semigroups. 
Notice that for the action of $\Tcal$ on $X_\ep$, 
$E(X_\ep,\Tcal) = E(X,T)$ as subsets of $X^X$ (as $t_{d+1}t_\ep \in T$). 

 
Let now $u \in E(X,T)$ be any minimal idempotent. Then 
$\tilde u = (u,u,...,u) \in E(N,\Tcal)$.
Choose $v$ a minimal idempotent in the closed left ideal 
$E(N, \Tcal)\tilde u$. Then $v\tilde u = v$. We want to show that 
$\tilde uv = \tilde u$. 
Set, for $\ep \subset \{1,\dots,d\}, \ v_\ep = \pi_\ep^\bullet v$. 
Note that, as an element of $E(N,\Tcal)$ is determined by its projections, 
it suffices to show that for each $\ep, \ uv_\ep = u$.
Since for each $\ep$ the map $\pi_\ep^\bullet$ is a semigroup homomorphism, we have that $v_\ep u = v_\ep$ as $v\tilde u = v$. In particular we deduce that 
$v_\ep$ is an idempotent belonging to the minimal left ideal 
$E(X_\ep,T)u=E(X,T)u$ which contains $u$. This implies (see \cite[Exercise 1.23.2.(b)]{G}) that
$$
 uv_\ep  = u, 
$$
and it follows that indeed $\tilde u v = \tilde u$. Thus, $\tilde u$ is an element of the minimal left ideal 
$E(N,\Tcal)v$ which contains $v$, and therefore $\tilde u$ is a minimal idempotent of 
$E(N, \Tcal)$.

Now let $x \in X$ and let $u$ be a minimal idempotent in $E(X,T)$ with 
$ux = x$ (since $(X,T)$ is minimal there always exists such an idempotent). 
By the above argument, $\tilde u$ is also a minimal idempotent of $(N,\Tcal)$ which implies that 
$N = Q^{[d]}$, the orbit closure of $x^{[d]} = \tilde u x^{[d]}$, 
is $\Tcal$ minimal (see \cite[Exercise 1.26.2]{G}). 

\br

2.\ 
Given $x \in X$ we now let $N := Q_{x_*}^{[d]}$ and $\Tcal : = \Fcal_d$.
The proof of the minimality of the flow $(Q_{x_*}^{[d]},\Fcal_d)$ is almost verbatim the same,
except that here the claim that for $u$ a minimal idempotent in $E(X,T)$, the map 
$\tilde u = (u,u,...,u)$ ($2^d -1$ times) is in  $E(Q_{x_*}^{[d]}, \Fcal_d)$, is not that evident. 
However, as $u$ is an idempotent this fact follows from the following lemma
(with $p_1= \cdots = p_d =u$). 
 
\begin{lem}
Let 
$p_1,\dots , p_d \in E(X,T)$ and for $\ep =\{n_1,\dots, n_k\} \subset \{1,\dots,d\}$, 
with $n_1< \cdots < n_k$,
let $q_\ep = p_{n_k} \cdots p_{n_1}$.  
Then the map
$(q_\ep : \ep \subset  \{1,\dots,d\}, \ \ep\not=\emptyset)$ is an element of $E(Q_{x_*}^{[d]}, \Fcal_d)$.
\end{lem}
%

\begin{proof}
By induction on $d$, using the identity (\ref{tautheta}), or more specifically
$$
\tau^{[d]}_{(e, \dots,e ,t_d)} (x', x'') = ( x',\theta^{[d-1]}_{t_d}  x''),
$$
and the fact that right multiplication in $E(X,T)$
is continuous.
\end{proof}

 \br
 
 3.\
 We first reproduce the ingenious ``useful lemma" from \cite{SY}.
 
\begin{lem}\label{use}
If $(x^{[d-1]}, w) \in Q^{[d]}$ for some $x \in X$ and $w \in X^{[d-1]}$ and $(x^{[d-1]}, w)$ is an
$\Fcal_d$-minimal point, then $(x^{[d-1]}, w) \in Y_x^{[d]}$.
\end{lem}

\begin{proof}

Since $(x^{[d-1]}, w) \in Q^{[d]}$ 
it follows that $(x^{[d-1]}, w)$ is in the $\Gcal_d$-orbit closure of $x^{[d]}$,
i.e. there is a sequence 
$\{(t_{1k},  \dots, t_{dk}, t_{d+1k})\}_{k \in \N} \subset T^{d+1}$ 
such that
$$
\theta^{[d]}_{t_{d+1k}}\tau^{[d]}_{(t_{1k},  \dots, t_{dk})}(x^{[d]}) \to (x^{[d-1]}, w).
$$
Now
\begin{gather*}
\theta^{[d]}_{t_{d+1k}}\tau^{[d]}_{(t_{1k},  \dots, t_{dk})}(x^{[d]}) = \\
\theta^{[d]}_{t_{d+1k}}(\id^{[d-1]} \times \theta^{[d-1]}_{t_{dk}})
(\tau^{[d-1]}_{(t_{1k},  \dots, t_{d-1k})}(x^{[d-1]}),
\tau^{[d-1]}_{(t_{1k},  \dots, t_{d-1k})}(x^{[d-1]})) =\\
(\id^{[d-1]} \times \theta^{[d-1]}_{t_{dk}})
(\theta^{[d-1]}_{t_{d+1k}}\tau^{[d-1]}_{(t_{1k},  \dots, t_{d-1k})}(x^{[d-1]}),
\theta^{[d-1]}_{t_{d+1k}}\tau^{[d-1]}_{(t_{1k},  \dots, t_{d-1k})}(x^{[d-1]})),
\end{gather*}
and letting $a_k := \theta^{[d-1]}_{t_{d+1k}}\tau^{[d-1]}_{(t_{1k},  \dots, t_{d-1k})}(x^{[d-1]})$,
we have:
\begin{equation}\label{ak}
(\id^{[d-1]} \times \theta^{[d-1]}_{t_{dk}})(a_k,a_k) \to (x^{[d-1]}, w).
\end{equation}

Let 
\begin{gather*}
\pi_1 : (X^{[d]}, \Fcal_d) \to (X^{[d-1]}, \Fcal_d), \qquad	(x', x'') \mapsto x',\\
\pi_2 : (X^{[d]}, \Fcal_d) \to (X^{[d-1]}, \Fcal_d), \qquad    (x', x'') \mapsto x'',
\end{gather*}
be the projections to the first and last $2^{d-1}$ coordinates respectively. 
For $\pi_1$ we consider the action of the group $\Fcal_d$ on $X^{[d-1]}$ via	the	representation
$$
\tau^{[d]}_{(t_1,  \dots, t_d)} \mapsto \tau^{[d-1]}_{(t_1,  \dots, t_{d-1})},
$$  
and for $\pi_2$ the action is via the representation
$$
\tau^{[d]}_{(t_1,  \dots, t_d)} \mapsto \theta^{[d-1]}_{t_d} \tau^{[d-1]}_{(t_1,  \dots, t_{d-1})}.
$$

Denote the corresponding semigroup homomorphisms of enveloping semigroups by 
$$
\pi_i^\bullet : E(X^{[d]}, \Fcal_d) \to E(X^{[d-1]}, \Fcal_d),	\quad i = 1,2.
$$
Notice that for these actions of $\Fcal_d$ on $X^{[d-1]}$, as subsets of 
${X^{[d]}}^{X^{[d]}}$,
$$
\pi_1^\bullet ( E(X^{[d]}, \Fcal_d)) = E(X^{[d-1]}, \Fcal_{d-1})	
\ \ {\text{and}} \ \ 
\pi_2^\bullet (E(X^{[d]}, \Fcal_d)) = E(X^{[d-1]}, \Gcal_{d-1}).
$$
Thus for $p \in E(X^{[d]}, \Fcal_d)$ and $x \in X^{[d]}$, 
we have: 
$$
px = p(x',x'') = (\pi_1^\bullet (p)x', \pi_2^\bullet (p)x'').
$$

Now fix a minimal left ideal $L$ of $E(X^{[d]}, \Fcal_d)$. 
By (\ref{ak}) $a_k \to x^{[d-1]}$ and, 
since $(Q^{[d-1]},\Gcal_{d-1})$ is minimal, there exist 
$p_k \in L$ such that $a_k = \pi_2^\bullet (p_k)x^{[d-1]}$. 
Without loss of generality we assume that $p_k \to p \in L$. Then
$$
\pi_2^\bullet (p_k)x^{[d-1]} = a_k  \to x^{[d-1]}\ \ {\text{and}} \ \ 	
\pi_2^\bullet (p_k)x^{[d-1]} \to  \pi_2^\bullet (p)x^{[d-1]}.
$$ 
Hence
\begin{equation}\label{pF}
\pi_2^\bullet (p)x^{[d-1]} = x^{[d-1]}.
\end{equation}
Since $L$ is a minimal left ideal and $p \in L$ there exists a minimal idempotent 
$v \in J(L)$ such that $vp = p$. 
Then 
$$
\pi_2^\bullet(v)x^{[d-1]} = \pi_2^\bullet (v)\pi_2^\bullet (p)x^{[d-1]}
=\pi_2^\bullet (vp)x^{[d-1]} =\pi_2^\bullet (p)x^{[d-1]} = x^{[d-1]}.
$$
Let
$$
F = \mathfrak{G}(\ol{\Fcal_{d-1}(x^{[d-1]})}, x^{[d-1]}) 
= \{\al \in  vL : \pi_2^\bullet (\al)x^{[d-1]} = x^{[d-1]}\}
$$ 
be the Ellis group of the pointed flow $(\ol{\Fcal_{d-1}(x^{[d-1]})}, x^{[d-1]})$.
Then $F$ is a subgroup of the group $vL$. 
By (\ref{pF}), we have $ p \in F$ and
since $F$ is a group, we have $pFx^{[d]} = Fx^{[d]} 
\subset \pi_2^{-1}(x^{[d]})$.
Since $vx^{[d]} \in Fx^{[d]} = pFx^{[d]}$, there is some $x_0 \in  Fx^{[d]}$ such that 
$vx^{[d]} = px_0$. Set $x_k = p_k x_0$, then
\begin{equation}\label{xk}
x_k = p_k x_0 \to px_0 = vx^{[d]} = (\pi_1^\bullet (v)x^{[d-1]},x^{[d-1]}).
\end{equation}
As $x_0 \in  Fx^{[d]}$, it follows that $\pi_2(x_0) = x^{[d-1]}$, hence
$$	
\pi_2(x_k) = \pi_2(p_k x_0) = \pi_2^\bullet (p_k)\pi_2(x_0) 
= \pi_2^\bullet (p_k)x^{[d-1]} = a_k \to  x^{[d-1]}.
$$

Let $x_k = (b_k, a_k) \in \ol{\Fcal_{d}(x^{[d]})}$; 
then, by (\ref{xk}), 
$\lim b_k = \pi_1^\bullet (v)x^{[d-1]}$. 
By (\ref{ak}),  
$\theta^{[d-1]}_{t_{dk}}  a_k \to w$, 
hence
$$
(\id^{[d-1]} \times \theta^{[d-1]}_{t_{dk}})(b_k, a_k) = (b_k, \theta^{[d-1]}_{t_{dk}} a_k) 
\to (\pi_1^\bullet (v)x^{[d-1]}, w).
$$
Since 
$\id^{[d-1]} \times \theta^{[d-1]}_{t_{dk}} = \tau^{[d]}_{(e,\dots,e,t_{dk})}
\in \Fcal_d$ 
and 
$(b_k,a_k) \in  \ol{\Fcal_{d}(x^{[d]})}$, 
we have
\begin{equation}\label{v}
(\pi_1^\bullet (v)x^{[d-1]}, w) \in  \ol{\Fcal_{d}(x^{[d]})}.
\end{equation}
Since, by assumption, $(x^{[d-1]},w)$ is $\Fcal_d$ minimal,  there is some minimal
idempotent $u \in J (L)$ such that $u(x^{[d-1]}, w) = 
(\pi_1^\bullet (u)x^{[d-1]}, \pi_2^\bullet (u)w) = (x^{[d-1]}, w)$.
Since $u, v \in L$ are minimal idempotents in the same minimal left ideal $L$, we have $uv = u$. 
Thus
$u(\pi_1^\bullet (v)x^{[d-1]}, w) = 
(\pi_1^\bullet (u)\pi_1^\bullet (v)x^{[d-1]}, \pi_2^\bullet (u)w) = 
(\pi_1^\bullet (uv)x^{[d-1]}, w) = (\pi_1^\bullet (u)x^{[d-1]}, w) = (x^{[d-1]}, w)$.
By (\ref{v}), we have
$(x^{[d-1]}, w) \in   \ol{\Fcal_{d}(x^{[d]})}$ and the proof of the lemma is completed.	
\end{proof}

We are now ready to complete the proof of part (3) of the theorem. We assume by induction that
this assertion holds for every $ 1 \le j \le d-1$ and now, given $x \in X$, consider a minimal subflow
$Y$ of the flow $(Q_x^{[d]}, \Fcal_d)$. With notations as in the previous lemma, we observe that
$Y_1 = \pi_1(Y)$ is a minimal subflow of the flow  $(Q_x^{[d-1]}, \Fcal_{d-1})$ and therefore,
by the induction hypothesis $Y_1 = Y_x^{[d-1]} = \ol{\Fcal_{d-1} x^{[d-1]}}$.
But then for some $w \in  Q^{[d-1]}$ we have $(x^{[d-1]},w) \in Y$ and,
applying Lemma \ref{use}, we conclude that $(x^{[d-1]},w) \in Y_x^{[d]}$. Thus
 $Y = Y_x^{[d]}$ and the proof is complete. 
  
 \br
 
 4.\
 Let $2^{X^{[d]}}$ be the compact hyperspace consisting of the closed subsets of $X^{[d]}$
 equipped with the (compact metric) Vietoris topology. Let $\Phi : X \to  2^{X^{[d]}}$
 be the map $x \mapsto Y_x^{[d]}$. It is easy to check that this map is lower-semi-continuous
 (i.e. $x_i \to x \ \imp\ \liminf \Phi(x_i) \supset \Phi(x)$). It follows then that the set of continuity 
 points of $\Phi$ is a dense $G_\del$ subset $X_0 \subset X$ (see e.g. \cite{Ch}).
 Since the set $\Fcal_d\Del^{[d]}$ is dense in $Q^{[d]}$, it follows that at each point of
 $X_0$ we must have $Y_x^{[d]} = Q_x^{[d]}$.  
 \end{proof}

\br

\section{$RP^{[d]}$ is an equivalence relation}

In this section we outline the Shao-Ye proof that $RP^{[d]}$ is an equivalence relation.
We assume that $(X,T)$ is a minimal compact {\em metrizable} $T$-flow, where
$T$ is an abelian group. We fix a compatible metric $\rho$ on $X$.

\begin{defn}\label{def-eq}
The {\em regionally proximal relation of order $d$}
is the relation $RP^{[d]} \subset X^{[d]} \times X^{[d]}$ defined by the following condition:
$(x,y) \in RP^{[d]}$ iff for every $\del >0$ there is a pair $x', y' \in X$ and $(t_1,\dots,t_d)
\in T^d$ such that:
\begin{gather*}
\rho(x,x') < \del,  \ \  \rho(y,y') < \del \ \ {\text{and}} \\
\rho^{[d]}(\tau^{[d]}_{(t_1,\dots,t_d)}{x'}_*^{[d]},
\tau^{[d]}_{(t_1,\dots,t_d)}{y'}_*^{[d]}) : =
\sup \{\rho (t_\ep x', t_\ep y') : \ep \subset \{1,\dots,d\},\ \ep \not = \emptyset\} < \del.
\end{gather*}
For $d =1$ this relation is the classical {\em  regionally proximal relation}, see e.g. \cite{A}.
\end{defn}


A convenient characterization of $RP^{[d]}$ is provided by Host-Kra-Maass in \cite[Lemma 3.3]{HKM}. Among its implications one has the corollary that the relation $RP^{[d]}$ is preserved under factors (Corollary \ref{fac} below).

\begin{lem}\label{2.5}
Let $( X , T )$ be a minimal flow. Let $d \ge 1$ and $x , y \in X$ . Then
$(x,y) \in RP^{[d]}$ if and only if there is some $a_* \in  X_*^{[d]}$ 
such that $(x, a_* , y, a_* ) \in Q^{[d+1]}$. 
\end{lem}

\begin{proof}
Suppose first that $(x,y) \in RP^{[d]}$. Fix an arbitrary point $z \in X$. Then, given $\del >0$,
we first find a pair $x', y' \in X$ and $(t_1,\dots,t_d)\in T^d$ which satisfy the requirements in
Definition \ref{def-eq}, and then replace $x'$ by $sz$ and  $y'$ by $tsz$, with appropriate $s, t \in T$,
so that
$\rho(x,sz) < \del, \ \   \rho(y,tsz) < \del$ and 
$$
\rho^{[d]}(\tau^{[d]}_{(t_1,\dots,t_d)}{(sz)}_*^{[d]},
\tau^{[d]}_{(t_1,\dots,t_d)}{(tsz)}_*^{[d]}) < \del.
$$
Denoting $a_{*\del} = {(sz)}_*^{[d]}$ we have
\begin{gather*}
\tau^{[d+1]}_{(t_1,\dots,t_d,t)}(sz,  a_{*\del}, sz, a_{*\del})=
(\tau^{[d]}_{(t_1,\dots,t_d)}(sz,  a_{*\del}), \theta_t^{[d]}\tau^{[d]}_{(t_1,\dots,t_d)}(sz,  a_{*\del}))
\in Q^{[d+1]}
\end{gather*}
Now, chose a convergent subsequence to get
$$
\lim_{\del \to 0} \tau^{[d+1]}_{(t_1,\dots,t_d,t)}(sz,  a_{*\del}, sz, a_{*\del}) =
(x, a_* , y, a_* ) \in Q^{[d+1]}.
$$ 

Conversely, assume that there is some $a_* \in  X_*^{[d]}$ 
such that $(x, a_* , y, a_* ) \in Q^{[d+1]}$. 
Then, there exist sequences $x_n \in X$ and $F_n \in \Fcal_{d+1}$ such that
$$
F_n ((x_n)^{[d+1]}) \to (x, a_* , y, a_* ).
$$
Now $F_n$ has the form $F_n = (\tau^{[d]}_n, \theta^{[d]}_{t_n}\tau^{[d]}_n)$
with $t_n \in T$ and $\tau^{[d]}_n \in \Fcal_d$, so that $x_n \to x$ and $t_nx_n \to y$,
and it follows that $(x, y) \in RP^{[d]}$, as required.

\end{proof}

It follows directly from the definition that $RP^{[d]}$ is a symmetric and $T$-invariant relation.
It is also easy to see that it is closed. However, even for $d=1$ there are easy examples which show
that, in general, it need not be an equivalence relation (not being transitive).
The remarkable fact that when $(X,T)$ is minimal, and $T$ is abelian, the relation $RP^{[1]}$ is an
equivalence relation (and therefore coincides with the equicontinuous structure relation; 
i.e., the smallest closed invariant relation $S \subset X \times X$ such that the quotient flow
$(X/S,T)$ is equicontinuous) is due to Ellis and Keynes \cite{EK} (see also \cite{Mc}).

\begin{cor}\label{fac}
If $\pi : (X,T) \to (Y,T)$ is a homomorphism of minimal $T$-flows then
$$
(\pi \times \pi)(RP^{[d]}(X))  \subset RP^{[d]}(Y).
$$
\end{cor}

Equipped with Theorem \ref{min} we will now show that for every $d \ge 1$ the relation $RP^{[d]}$ is an equivalence relation. 
First we prove two more necessary and sufficient conditions on a pair $(x,y) \in X \times X$
to belong to $RP^{[d]}$.

\begin{prop}\label{eq}
Let $( X , T )$ be a minimal flow and $d \ge 1$. 
The following conditions are equivalent:
\begin{enumerate}
\item
$ (x, y) \in RP^{[d]}$.
\item
$ (x, y, y,\dots, y) = (x, y_*^{[d+1]}) \in Q^{[d+1]}$.
\item
$ (x, y, y, \dots, y) = (x, y_*^{[d+1]}) \in  \ol{\Fcal_{d+1}(x^{[d+1]})}$.
\end{enumerate}
\end{prop}

\begin{proof}

(3) $\imp$ (2) is obvious. The implication (2) $\imp$ (1) follows from Lemma \ref{2.5}. 
Thus it suffices to show that (1) $\imp$ (3).
Let $(x,y) \in RP^{[d]}$, then by Lemma \ref{2.5}, there is some $a	\in X^{[d]}$ 
such that $(x,a_* ,y,a_* ) \in Q^{[d+1]}$. 
Observe that $(y,a_*) \in Q^{[d]}$. 
By Theorem \ref{min}.(3), there is a sequence $\{F_k\} \subset \Fcal_d$ such that 
$F_k(y,a_*) \to  y^{[d]}$. 
Hence
$$
(F_k \times F_k) (x,a_*,y,a_*) \to (x, y_*^{[d]}, y, y_*^{[d]})=(x,y_*^{[d+1]}).
$$
Since
$F_k \times F_k \in \Fcal_{d+1}$ and $(x,a_*,y,a_*) \in Q^{[d+1]}$,
 we have that $(x, y_*^{[d+1]}) \in Q^{[d+1]}$. 
By Theorem \ref{min}.(2), $y^{[d+1]}$ is $\Fcal_{d+1}$-minimal. It follows that $(x, y_*^{[d+1]})$ 
is also $\Fcal_{d+1}$-minimal. 
Now $(x, y_*^{[d+1]}) \in Q^{[d+1]}[x]$ is $\Fcal_{d+1}$-minimal and by Theorem \ref{min}.(3), 
$\ol{\Fcal_{d+1}(x^{[d+1]})}$ is the unique $\Fcal_{d+1}$-minimal subset in 
$Q^{[d+1]}[x]$. Hence we have that $(x, y_*^{[d+1]}) \in \ol{ \Fcal_{d+1}(x^{[d+1]})}$, 
and the proof is completed.	 
\end{proof}

As an easy consequence of Proposition \ref{eq} we now have the following theorem.
\begin{thm}\label{RP}
Let $(X,T)$ be a minimal metric flow, where $T$ an abelian group, and $d \ge 1$. 
Then $RP^{[d]}$ is an equivalence relation.
\end{thm}

\begin{proof}
It suffices to show the transitivity, i.e. if $(x, y), (y, z) \in RP^{[d]}$, then 
$(x, z) \in RP^{[d]}(X)$. Since $(x, y), (y, z) \in RP^{[d]}$, by Proposition \ref{eq} we have
$$
(y, x, x,\dots , x), (y, z ,z ,\dots , z) \in \ol{\Fcal_{d+1}(y^{[d+1]})}.
$$
By Theorem \ref{min}.(2) 
$(\ol{\Fcal_{d+1}(y^{[d+1]})}, \Fcal_{d+1})$ is minimal, whence
$$
(y, z, z,\dots , z) \in \ol{\Fcal_{d+1}(y, x, x, \dots, x)}.
$$	
Finally, as $\Fcal_{d+1}$ acts as the identity on the $\emptyset$-coordinate,
it follows that also 
$$
(x, z, z, \dots , z)	\in	\ol{\Fcal_{d+1}(x^{[d+1]})}.
$$	
By	Proposition \ref{eq} again, $(x, z) \in RP^{[d]}$.	
\end{proof}

\begin{rmk}
From Proposition \ref{eq} we deduce that in the definition of the 
regionally proximal relation of order $d$
the point $x'$ can be replaced by $x$. More precisely, a pair  $(x, y) \in X \times X$
is in $RP^{[d]}$ if and only if 
for every $\del >0$ there is a point $y' \in X$ and $(t_1,\dots,t_d)
\in T^d$ such that:
\begin{gather*}
\rho(y, y') < \del \ \ {\text{and}} \\
\rho^{[d]}(\tau^{[d]}_{(t_1,\dots,t_d)}{x}_*^{[d]},
\tau^{[d]}_{(t_1,\dots,t_d)}{y'}_*^{[d]}) : =
\sup \{\rho (t_\ep x, t_\ep y') : \ep \subset \{1,\dots, d\},\ \ep \not = \emptyset\} < \del.
\end{gather*}
Again for $d =1$ this is a well known result (see \cite{V} and \cite{Mc}).
\end{rmk}

\br

Let us conclude with the following remark.
It is not hard to see that the proximal relation $P \subset X \times X$ 
is a subset of $RP^d$ for each $d \ge 1$ (see Proposition 3.1 in \cite{HKM}).
Thus for every $d \ge 1$ the quotient flow $X/ PR^{[d]}$ is a minimal distal flow.
Of course the main result of Host, Kra and Maass in this work \cite{HKM} is the fact that, for $T = \Z$, this minimal distal factor flow
is the maximal factor of $(X,T)$ which is a {\em system of order $d-1$}; i.e.,
a $T$-flow which is an inverse limit of $(d-1)$-step minimal $T$-nilflows.
In turn, the results in \cite{HKM} are based on the profound analogous ergodic theoretical
theorems obtained by Host and Kra in \cite{HK}.

\end{document}